\documentclass[11pt,a4paper,notitlepage,oneside]{amsart}
\usepackage{etoolbox}
\usepackage{microtype} 
\usepackage{amsmath}
\usepackage{amsthm}
\usepackage{braket}
\usepackage{pdfpages}
\usepackage[psamsfonts]{amssymb}
\usepackage[T1]{fontenc}
\usepackage[english]{babel}
\usepackage[utf8]{inputenc}
\usepackage{enumitem}
\usepackage{mathtools}
\usepackage{caption,subfig} 
\usepackage{rotating} 
\usepackage{scalerel}
\usepackage{relsize} 
\usepackage{cancel} 

\usepackage[autostyle,italian=guillemets]{csquotes}
\usepackage[style=alphabetic,firstinits=true,maxbibnames=99,backend=bibtex]{biblatex}
\bibliography{bibliography}
\allowdisplaybreaks 

\usepackage{pgf,tikz,pgfplots}
\usepackage{mathrsfs}
\usetikzlibrary{arrows}

\usepackage[hyphenbreaks]{breakurl}

\usepackage{fancyhdr}
\usepackage{tikz-cd}
\usepackage{amscd}
\usepackage[a4paper, bottom=3cm, left=3cm,right=3cm]{geometry}
\usepackage{hyperref}
\hypersetup{hidelinks} 

\usepackage{subfiles}
\numberwithin{equation}{section} 
\theoremstyle{definition}
\newtheorem{defi}{Definition}[section]
\newtheorem{ex}[defi]{Example}

\theoremstyle{plain}
\newtheorem{thm}[defi]{Theorem}
\theoremstyle{plain}
\newtheorem*{thm*}{Theorem}
\theoremstyle{plain}

\theoremstyle{plain}

\theoremstyle{plain}
\newtheorem{prop}[defi]{Proposition}
\theoremstyle{plain}
\newtheorem{lemma}[defi]{Lemma}
\theoremstyle{plain}
\newtheorem{cor}[defi]{Corollary}
\theoremstyle{plain}

\theoremstyle{plain}
\newtheorem*{question*}{Question}

\theoremstyle{definition}
\newtheorem{rem}[defi]{Remark}
\theoremstyle{thm}

\definecolor{ao(english)}{rgb}{0.0, 0.5, 0.0}

\DeclareMathOperator{\CH}{\rm CH}

\DeclareMathOperator{\GL}{\rm GL}

\DeclareMathOperator{\SO}{\rm SO}

\DeclareMathOperator{\supp}{\rm supp}
\DeclareMathOperator{\Res}{\rm Res}

\newcommand{\NN}{\mathbb{N}}
\newcommand{\ZZ}{\mathbb{Z}}
\newcommand{\RR}{\mathbb{R}}
\newcommand{\QQ}{\mathbb{Q}}
\newcommand{\CC}{\mathbb{C}}

\DeclareMathOperator{\Gr}{\rm Gr}

\newcommand{\vect}[1]{\boldsymbol{#1}}

\newcommand{\bbcomp}{\overline{X}^{\raisebox{-1pt}{$\scriptscriptstyle BB$}}}
\newcommand{\halfint}{\Lambda}

\newcommand{\primHeegner}[1]{H^{\rm prim}_{#1}}

\newcommand{\seq}[2]{( #1 )_{#2}} 

\newcommand{\projmodn}{\mathcal{D}^+_n}

\DeclareMathOperator{\Vol}{\rm Vol}
\DeclareMathOperator{\vol}{\rm vol}

\DeclareMathOperator{\bigO}{\rm O}

\DeclareMathOperator{\SU}{\rm SU}
\newcommand{\genosub}[2]{\mathcal{G}^{\mathrm{ort}}_{#1}(#2)}
\newcommand{\coneosub}[2]{\mathcal{C}^{\mathrm{ort}}_{#1}(#2)}
\newcommand{\genspec}[1]{\mathcal{G}^{\mathrm{sp}}_{n-2}(#1)}
\newcommand{\conespec}[1]{\mathcal{C}^{\mathrm{sp}}_{n-2}(#1)}

\def\be{\begin{equation}}
\def\ee{\end{equation}}
\def\bes{\begin{equation*}}
\def\ees{\end{equation*}}
\def\ba{\be\begin{aligned}}
\def\ea{\end{aligned}\ee}
\def\bas{\bes\begin{aligned}}
\def\eas{\end{aligned}\ees}

\newcommand{\myblack}{black}

\author{Riccardo Zuffetti}
\title{Cones of orthogonal Shimura subvarieties and equidistribution}
\address{
Fachbereich Mathematik, Technische Universit\"at Darmstadt, Schlossgartenstrasse 7,
D–64289 Darmstadt, Germany\\
}
\email{zuffetti@mathematik.tu-darmstadt.de, riccardo.zuffetti@gmail.com}

\begin{document}
	\maketitle
	\begin{abstract}
	Let~$X$ be an orthogonal Shimura variety, and let~$\coneosub{r}{X}$ be the cone generated by the cohomology classes of orthogonal Shimura subvarieties in~$X$ of dimension~$r$.
	We investigate the asymptotic properties of the generating rays of~$\coneosub{r}{X}$ for large values of~$r$.
	They accumulate towards rays generated by wedge products of the Kähler class of~$X$ and the fundamental class of an orthogonal Shimura subvariety.
	We also compare~$\coneosub{r}{X}$ with the cone generated by the special cycles of dimension~$r$.
	The main ingredient to achieve the results above is the equidistribution of orthogonal Shimura subvarieties.
	\end{abstract}
	\tableofcontents
	
	\emergencystretch 3em 
	
	\section{Introduction}
	In recent years the cones of effective cycles of codimension higher than~$1$ have attracted increasing interest~\cite{fulgconeprojbund} \cite{chencosk} \cite{effcone2cymodc} \cite{lisphericalvar} \cite{coleotch} \cite{effsymprod} \cite{gruhu}.
	A clear description of these cones is available only for few examples, making the overall view far from being well-understood.
	
%
	In~\cite{zufcones} we considered cones generated by the \emph{special cycles} of codimension~$2$ on orthogonal Shimura varieties.
	We proved that the rays generated by these effective cycles accumulate towards infinitely many rays, the latter generating a rational polyhedral cone.
	
	In this article we investigate the properties of the larger cones generated by (the cohomology classes of) the irreducible components of special cycles.
	We usually refer to these irreducible components as \emph{orthogonal Shimura subvarieties}.
	
	The codimension~$1$ analogues of the latter cones were firstly considered by Bruinier and Möller in~\cite[Section~$4$]{brmo}, where they proved that the cone of irreducible components of special divisors is polyhedral.
	The proof relies on showing that the rays spanned by these codimension~$1$ subvarieties accumulate towards a unique internal ray of the cone.
	
	In this article we compute all rays towards which the rays generated by orthogonal Shimura subvarieties of codimension~$2$ may accumulate, together with a partial generalization in higher codimension.
	These results are achieved using the equidistribution properties owned by such subvarieties.
	
	To state our results more precisely, we need to introduce some notation.
	Let~$(V,q)$ be an indefinite rational quadratic space of signature~$(n,2)$.
	We denote by~$G$ the linear algebraic group of isometries~$\SO(V,q)$.
	For every arithmetic lattice~$\Gamma\subset G(\QQ)$, and every maximal compact subgroup~$K$ of~$G(\RR)$, we consider the double quotient~$X=\Gamma\backslash G(\RR)/K$.
	It admits a unique structure of algebraic variety by the Theorem of Baily and Borel.
	We refer to such arithmetic varieties as orthogonal Shimura varieties.
	One of their interesting features is that they admit many algebraic cycles, which may be constructed by immersion in~$X$ of Shimura varieties of smaller dimension.
	
	Let~$(V',q')$ be an indefinite rational quadratic subspace of signature~$(r,2)$ in~$(V,q)$, and let~$H$ be the~$\QQ$-subgroup~$\SO(V',q')$ of~$G$.
	We say that the subvariety~$Z=\Gamma\backslash\Gamma H(\RR) K/K$ of~$X$ is an \emph{orthogonal Shimura subvariety}.
	It is the immersion in~$X$ of an orthogonal Shimura variety arising from~$H$.
	
	We denote by~$\genosub{r}{X}$ the set of cohomology classes in~$H^{2(n-r)}(X,\RR)$ of~$r$-dimensional orthogonal Shimura subvarieties.
	\begin{defi}
	The cone~$\coneosub{r}{X}$ of the~$r$-dimensional orthogonal Shimura subvarieties of~$X$ is the cone in~$H^{2(n-r)}(X,\RR)$ generated by~$\genosub{r}{X}$.	
	\end{defi}
	Following the wording of~\cite[Section~$4$]{zufcones}, we say that a ray~$R$ in~$H^{2(n-r)}(X,\RR)$ is an \emph{accumulation ray of~$\coneosub{r}{X}$ with respect to the set of generators~$\genosub{r}{X}$} if there exists a sequence of pairwise different cohomology classes~$\big([Z_j]\big)_{j\in\NN}$ in~$\genosub{r}{X}$ such that 
	\bes
	\RR_{\ge0}\cdot [Z_j] \longrightarrow R,\qquad\text{when $j\longrightarrow\infty$},
	\ees
	where we denote by~$\RR_{\ge0}\cdot [Z_j]$ the ray generated by~$[Z_j]$.
	If not written otherwise, every accumulation ray of~$\coneosub{r}{X}$ is with respect to~$\genosub{r}{X}$.
	
	The \emph{accumulation cone of~$\coneosub{r}{X}$ with respect to~$\genosub{r}{X}$} is the cone in~$H^{2(n-2)}(X,\RR)$ generated by the accumulation rays of~$\coneosub{r}{X}$.
	
	We denote by~$\omega$ any~$G(\RR)$-invariant Kähler form of the Hermitian symmetric domain~$G(\RR)/K$.
	The form~$\omega$ induces a cohomology class~$[\omega]$ on~$X=\Gamma\backslash G(\RR)/K$, for every arithmetic subgroup~$\Gamma$ of~$G(\QQ)$.
	
	As recalled above, Bruinier and Möller~\cite{brmo} proved that the only accumulation ray of~$\coneosub{n-1}{X}$ is~$\RR_{\ge0}\cdot[\omega]$.
	In higher codimension the geometry of the accumulation rays is more interesting, as shown by the following result.
	\begin{thm}\label{thm;acconeortshsubv}
	Let~$X$ be an orthogonal Shimura variety of dimension~$n$, and let~$r>(n+1)/2$.
	If~$[Z]$ is a non-zero cohomology class arising from an orthogonal Shimura subvariety~$Z$ of dimension~$r'>r$ in~$X$, then the ray~$\RR_{\ge0}\cdot [\omega]^{r'-r}\wedge[Z]$ is an accumulation ray of~$\coneosub{r}{X}$.
	\end{thm}
	
	It is well-known that the cohomology classes of compact subvarieties in a Kähler manifold are non-zero.
	This result is not available in the literature for (non-compact) orthogonal Shimura subvarieties; see e.g.~\cite{fm;boundarybe} for the non-vanishing of infinitely many classes of special cycles.
	For this reason, the hypothesis appearing in Theorem~\ref{thm;acconeortshsubv} that the class~$[Z]$ is non-zero is not trivial.
	
	An example in which cohomology classes of orthogonal Shimura subvarieties do not vanish is provided in codimension~$1$ by~\cite[Section~$4$]{brmo}, as we briefly recall.	
	We restrict to unimodular lattices so that we may compare cones of orthogonal Shimura subvarieties with the cones generated by special cycles considered in~\cite{zufcones}.
	
	Let~$L$ be an even lattice\textcolor{\myblack}{, i.e.\ a free~$\ZZ$-module of finite rank equipped with a symmetric~$\ZZ$-valued bilinear form~$(\cdot{,}\cdot)$ such that~$(\lambda,\lambda)\in 2\ZZ$ for all~$\lambda \in L$, of signature~$(n,2)$.
	We assume that~$L$ is \emph{unimodular}, namely that~$L$ equals its dual lattice.}
	If~$X$ is an orthogonal Shimura variety arising from~$G=\SO(L\otimes\QQ)$ and~$\Gamma=\SO^+(L)$, then every orthogonal Shimura subvariety of codimension~$1$ in~$X$ induces a non-zero cohomology class in~$H^2(X,\RR)$.
	
	Under these assumptions, we may refine Theorem~\ref{thm;acconeortshsubv} in codimension~$2$, providing the following complete classification of the accumulation rays of~$\coneosub{n-2}{X}$.
	\begin{cor}\label{cor;acconeortshsubv}
	Let~$X=\SO^+(L)\backslash G(\RR)/K$ be an orthogonal Shimura variety of dimension~$n>5$ arising from an even unimodular lattice~$L$.
	The accumulation rays of~$\coneosub{n-2}{X}$ are~$\RR_{\ge0}\cdot[\omega]^2$ and the rays~$\RR_{\ge0}\cdot[\omega]\wedge[Z]$, for every orthogonal Shimura subvariety~$Z$ of codimension~$1$.
	\end{cor}
	
	We denote by~$\genspec{X}$ the set of cohomology classes of codimension~$2$ special cycles of~$X$.
	The cone~$\conespec{X}$ of codimension~$2$ special cycles of~$X$ is the cone generated by~$\genspec{X}$ in~$H^{4}(X,\RR)$.
	
	By means of the classification of the accumulation rays of~$\conespec{X}$ in~\cite[Section~$8$]{zufcones}, and the analogous classification for~$\coneosub{n-2}{X}$ provided by Corollary~\ref{cor;acconeortshsubv}, we may deduce the following result.
	We denote by~$M^k_1$ the space of elliptic modular forms of weight~$k$.
	
	\begin{cor}\label{cor;acconeortshsubvfoll}
	Let~$X=\SO^+(L)\backslash G(\RR)/K$ be an orthogonal Shimura variety of dimension~$n>5$ arising from an even unimodular lattice~$L$.
	The accumulation cone of~$\coneosub{n-2}{X}$ equals the accumulation cone of~$\conespec{X}$.
	In particular, the accumulation cone of~$\coneosub{n-2}{X}$ is pointed, rational, polyhedral and of dimension~$\dim M^{1+n/2}_1$.
	\end{cor}
	
	Theorem~\ref{thm;acconeortshsubv} and Corollary~\ref{cor;acconeortshsubv} are proved by means of equidistribution results on orthogonal Shimura subvarieties, as we briefly illustrate.
	The Kähler form~$\omega$ induces probability measures~$\nu_X$ and~$\nu_Z$, respectively on~$X$ and on any orthogonal Shimura subvariety~$Z$.
	Let~$\seq{Z_j}{j\in\NN}$ be a sequence of \textcolor{\myblack}{pairwise different} orthogonal Shimura subvarieties of positive dimension~\textcolor{\myblack}{$r$}.
	In Section~\ref{sec;seqincoho} we show that there exist an orthogonal Shimura subvariety~$Z$ of dimension~$r'>r$ in~$X$ and a subsequence~$\seq{Z_s}{s}$, such that the subvarieties~$Z_s$ \emph{equidistribute} in~$Z$, i.e.\ the sequence of probability measures~$\nu_{Z_s}$ weakly converges to~$\nu_Z$.
	This is a refinement of~\cite[\textcolor{\myblack}{Th\'eor\`eme}~$1.2$]{cloull;equid1}.
	It implies by~\cite{kozmau} and~\cite{tatho} that
	\be\label{eq;mainresu}
	\frac{[Z_j]}{\Vol(Z_j)}\xrightarrow[\,j\to\infty\,]{} \frac{r!}{r'!}\cdot[\omega]^{r'-r}\wedge\frac{[Z]}{\Vol(Z)}\qquad\text{in $H^{2(n-r)}(X,\QQ)\cap H^{n-r,n-r}(X)$}.
	\ee
%
	
	\textcolor{\myblack}{The equidistribution results on orthogonal Shimura subvarieties are proved in Section~\ref{sec:equidres}.
	In Section~\ref{sec;seqspecialcycles} we apply~\eqref{eq;mainresu} to deduce Theorem~\ref{thm;acconeortshsubv} and Corollary~\ref{cor;acconeortshsubv}.
	A comparison of the cones~$\coneosub{n-2}{X}$ and~$\conespec{X}$ is also given.}
	
	\subsection*{Acknowledgments}
	We are grateful to Martin Möller for suggesting this project and for his encouragement.
	His office door was always open for discussions.
	We would like to thank Jan Bruinier, Jens Funke \textcolor{\myblack}{and Salim Tayou} for useful conversations on this topic.
	\textcolor{\myblack}{We also wish to thank the anonymous referee for a thorough review and helpful suggestions on the final version of the present paper.}
	This work is a result of our PhD~\cite{zufthesis}, which was founded by the LOEWE research unit ``Uniformized Structures in Arithmetic and Geometry'', and by the Collaborative Research Centre TRR~326 ``Geometry
and Arithmetic of Uniformized Structures'', project number~444845124.
	\section{Orthogonal Shimura varieties and special subvarieties}\label{sec;Shimvar}
	Throughout this paper we denote by~$G$ the linear algebraic group of isometries~$\SO(V,q)$ associated to some indefinite rational quadratic space~$(V,q)$ of signature~$(n,2)$.
	The Hermitian symmetric domain associated to~$G$ is the Kähler manifold arising as the quotient~$\widetilde{X}=G(\RR)/K$, for some maximal compact subgroup~$K$ of~$G(\RR)$.
	Up to isomorphism, the choice of~$K$ does not affect~$\widetilde{X}$.
	For this reason, we may suppose~$K$ to be the standard maximal compact subgroup ${\rm S}\big({\rm O}(n)\times {\rm O}(2)\big)$.
	It is well-known that~$\widetilde{X}$ can be realized as the Grassmannian $\Gr(V)$ of negative definite $2$-panes in~$V\otimes\RR$.
	
	An \emph{arithmetic subgroup} $\Gamma$ of $G(\QQ)$ is a subgroup of $G(\QQ)\cap G(\RR)^+$, where~$G(\RR)^+$ is the connected component of the identity of $G(\RR)$ with respect to the Euclidean topology, such that $\Gamma\cap G(\ZZ)$ is of finite index in $G(\ZZ)$ and $\Gamma$.
	\begin{defi}
	A \emph{(connected) orthogonal Shimura variety} is the double quotient~$X=\Gamma\backslash G(\RR)/K$ arising from some arithmetic lattice $\Gamma$ of $G(\QQ)$.
	\end{defi}
	By the Theorem of Baily and Borel, there exists a unique algebraic structure on any quotient~$X=\Gamma\backslash G(\RR)/K$ as above.
	With such a structure, the variety~$X$ is either projective or quasi-projective. The former case can happen only when $n<3$ \textcolor{\myblack}{by Meyer's Theorem; see e.g.~\cite[Page~$4$]{Kudla;speccycl}.}
	\begin{rem}
	Orthogonal Shimura varieties are usually defined with respect to congruence subgroups.
	Since the results on equidistribution that we are going to apply in this paper work for more general arithmetic subgroups as well, we do not require~$\Gamma$ to be of congruence.
	\end{rem}
	In this paper we deal with certain \emph{special subvarieties} of orthogonal Shimura varieties, defined below.
	The terminology comes from the fact that these subvarieties can be considered as immersions in~$X$ of Shimura varieties of smaller dimension; see e.g.~\cite[Section~$3.3$]{ull;equidsurv}.
	\begin{defi}\label{def;specialsubvar}
	Let $X=\Gamma\backslash G(\RR)/K$ be an orthogonal Shimura variety.
	If $H$ is a~$\QQ$-algebraic subgroup of~$G$ which induces an inclusion of Hermitian symmetric domains
	\bes
	\widetilde{Y}=H(\RR)/(K\cap H(\RR))\hookrightarrow G(\RR)/K,
	\ees
	we say that the immersion of $(\Gamma\cap H(\RR)^+)\backslash\widetilde{Y}$ in $X$ is a \emph{special subvariety}.
	
	If a special subvariety $Y$ of $X$ arises from a~$\QQ$-subgroup $H$ of $G$ such that $H=\SO(V',q')$, for some rational quadratic subspace~$(V',q')$ of signature~$(n',2)$ in~$(V,q)$, where~$n'\ge 1$, we say that~$Y$ is an \emph{orthogonal Shimura subvariety}.
	\end{defi}
	\textcolor{\myblack}{As clarified below, the orthogonal Shimura subvarieties of~$X$ are not the only possible Shimura subvarieties of~$X$.}
	\begin{rem}\label{rem;fromFiori}
	\textcolor{\myblack}{\textcolor{\myblack}{Fiori~\cite{fiori} classified the Shimura subvarieties of~$X$ by proving that they may arise only from (restriction of scalars of) orthogonal or unitary algebraic groups.}
	We call the former kind \emph{Shimura subvarieties of orthogonal type} of~$X$.
	They are the Shimura subvarieties which arise from a~$\QQ$-subgroup~$H$ of~$G$ of the form $H=\Res_{F/\QQ}\SO(U,q_U)$, for some quadratic space~$(U,q_U)$ defined over a totally real extension~$F$ of~$\QQ$, of signature~$(\ell,2)$ at one place and positive definite at all other places.}
	By \cite[Construction 3.5]{fiori}, the inclusion of groups~${H\hookrightarrow G=\SO(V,q)}$ factors \textcolor{\myblack}{through} base change to $\RR$ as follows, with surjective projection onto the first factor~$\SO(\ell,2)$.
	\bes
	\begin{tikzcd}[column sep=small]
	H(\RR)\arrow{rr}\arrow{dr} && \SO(n,2)\cong G(\RR)\\
	& {\SO(\ell,2)\times\SO(\ell+2)\times\cdots\times\SO(\ell+2)}\arrow[ur] & 
	\end{tikzcd}
	\ees
	The orthogonal Shimura subvarieties of Definition \ref{def;specialsubvar} are the special subvarieties of orthogonal type as above, with $F=\QQ$.
	\end{rem}
	
	\begin{rem}\label{rem;fioridifShimsubvar}
	\textcolor{\myblack}{There are special subvarieties of~$X$, not of orthogonal type, which arise from unitary subgroups of~$G$}; see~\cite{fiori}.
	The Hermitian symmetric domain arising from~$\SU(m,1)$ is the \emph{complex hyperbolic $m$-space}~$\mathbb{B}^m$.
	Since all Hermitian symmetric domains contained in~$\mathbb{B}^m$ are complex hyperbolic subspaces, as proved e.g.~in~\cite[Proposition~$2.3$]{bfms}, the special subvarieties of orthogonal type in~$X$ are the only special subvarieties which may contain other special subvarieties of orthogonal type.
	\end{rem}
	\begin{lemma}\label{lemma:gaghk}
	Let $X=\Gamma\backslash G(\RR)/K$ be an orthogonal Shimura variety, and let~$H$ be the group of isometries~$\SO(W,q_W)$ of some rational quadratic subspace $(W,q_W)$ of signature~$(r,2)$ in~$(V,q)$, with $1\le r\le n$.
	Every orthogonal Shimura subvariety of $X$ of dimension $r$ is of the form~${\Gamma\backslash\Gamma gH(\RR)K/K}$ for some~$g\in G(\RR)$.
	\end{lemma}
	\begin{proof}
	Let $\widetilde{X}=G(\RR)/K$ be the Hermitian symmetric domain attached to $G$.
	We realize~$\widetilde{X}$ as the Grassmannian~$\Gr(V)$ of negative definite $2$-planes in~$V_\RR=V\otimes\RR$.
	Let $Z$ be an orthogonal Shimura subvariety of~$X$ of dimension~$r$, and let~$H'=\SO(V',q')$ be the~$\QQ$-algebraic subgroup of~$G$ such that~$Z$ is the immersion in~$X$ of~$\Gamma_{H'}\backslash H'(\RR)/K_{H'}$, where~$\Gamma_{H'}=\Gamma\cap H'(\RR)^+$, $K_{H'}=K\cap H'(\RR)$, and~$(V',q')$ is a rational quadratic subspace of signature~$(r,2)$ in~$(V,q)$.	
	The Hermitian symmetric domain~$\widetilde{Z}=H'(\RR)/K_{H'}$ associated to~$H'$ embeds into~$\widetilde{X}$, and it may be realized as the Grassmannian~$\Gr(V')$.
	
	The real quadratic subspaces~$W_\RR$ and~$V'_\RR$ of~$V_\RR$ have the same dimension and signature, hence there exists an isometry~${f\colon W_\RR\to V'_\RR}$.
	By Witt's Theorem, the isometry~$f$ extends to an isometry~${g\in\bigO(V_\RR)}$ such that~${g|_{W_\RR}=f}$.
	Up to composing~$g$ with a reflection with respect to a hyperplane of~$V_\RR$ containing~$W_\RR$, we may assume that~${g\in G(\RR)}$.
	Since~$g$ acts on~$\Gr(V)$ mapping~$\Gr(W)$ to~$\Gr(V')$, we deduce that~${gH(\RR)/K_H=H'(\RR)/K_{H'}}$.
	If we consider the immersion in $X$ of $gH(\RR)/K_H$, we deduce that
	\bes
	\Gamma\backslash\Gamma gH(\RR)K/K=\Gamma\backslash\Gamma H'(\RR)K/K=Z.\qedhere
	\ees
	\end{proof}
	Following the wording of~\cite{cloull;equid1}, we introduce the so-called strongly special subvarieties.
	As we will see in~Section~\ref{sec:equidres}, they have good equidistribution properties in~$X$.
	\begin{defi}\label{def;stronglyspec}
	A special subvariety of an orthogonal Shimura variety is said to be \emph{strongly special} if it arises from a semisimple $\QQ$-subgroup $H$ that is not contained in any proper parabolic $\QQ$-subgroup of $G$.
	\end{defi}
	We conclude this section by showing that every orthogonal Shimura subvariety is strongly special.
	
	\begin{prop}\label{prop;everyosv>2isstsp}
	Let $X$ be an orthogonal Shimura variety.
	Every orthogonal Shimura subvariety of~$X$ is strongly special.
	\end{prop}
	\begin{proof}
	Let~$(V,q)$ be a rational quadratic space of signature $(n,2)$ such that~$G$ equals $\SO(V,q)$, and such that $X=\Gamma\backslash G(\RR)/K$ for some arithmetic subgroup~$\Gamma$ of~$G$.
	 Let $Z$ be an orthogonal Shimura subvariety of $X$ of dimension~$r>0$.
	 It arises from a subgroup~${H=\SO(V',q')}$ of~$G$, for some rational quadratic subspace $(V',q')$ of signature $(r,2)$ in~$(V,q)$.
	We may consider~$H$ as a subgroup of~$G$ via the inclusion given by extending every isometry in $\SO(V',q')$ as the identity over $V'^\perp$.
	Equivalently, the group~$H$ is identified with the pointwise stabilizer of~$V'^\perp$ with respect to the action of~$\SO(V,q)$.
	
	Since $\dim(V)\ge 3$, the parabolic $\QQ$-subgroups of~$G$ are stabilizer subgroups of isotropic flags in~$V$; see e.g.~\cite[Theorem~T.$3.9$]{con;AgGrII}.
	We recall that a flag~$F$ in~$V$ is an increasing chain of non-zero proper subspaces of~$V$, denoted as
	\bes
	F=\{F_1\subsetneqq \dots \subsetneqq F_m\},\quad\text{for some $m>0$.}
	\ees
	A flag $F$ is said to be \emph{isotropic} if each $F_j$ is totally isotropic in $V$.
	We say that a subgroup of $G$ stabilizes the flag $F$ if it preserves every subspace $F_j$, where $j=1,\dots,m$.
	
	Suppose that $Z$ is not strongly special.
	This means that there exists a parabolic~$\QQ$-subgroup~$P$ of~$G$ containing~$H$.
	As previously remarked, every parabolic subgroup of~$G$ is the stabilizer of an isotropic flag of~$V$.
	We denote by $F$ the isotropic flag stabilized by~$P$.
	Since the Witt index of $(V,q)$ is at most~$2$, the maximal isotropic subspaces of $V$ have dimension at most $2$.
	Therefore, the isotropic flag~$F$ is either of the form~${F=\{F_1\subsetneqq F_2\}}$, or~$F=\{F_1\}$ with~$\dim(F_1)=1,2$.
	
	To conclude the proof it is enough to show that~$H$ does not stabilize any totally isotropic subspace $F_1\subset V$ of dimension $1$ or $2$.
	\textcolor{\myblack}{There are only finitely many orbits under the action of~$\SO(V',q')$ of the proper isotropic subspaces of~$V'$ of the same dimension as~$F_1$.}
	The orbits are actually at most~$2$ by~\cite[Proposition~T.$3.7$]{con;AgGrII}.
	Since whenever~$(V',q')$ is isotropic there is an infinite number of proper isotropic subspaces of~$V'$, we may assume that~$F_1\cap V'=\textcolor{\myblack}{\{0\}}$. 
	
	We begin with the case of~$\dim(F_1)=1$.
	Let $u$ be a basis vector of $F_1$, and let $\pi_{V'}$ (resp.~$\pi_{V'^\perp}$) be the projection on the first (resp.\ second) factor arising from the orthogonal decomposition $V=V'\oplus V'^\perp$.
	Since $u=\pi_{V'}(u)+\pi_{V'^\perp}(u)$ and $q(u)=0$, then
	\bes
	0=q\big(\pi_{V'}(u)\big)+q\big(\pi_{V'^\perp}(u)\big).
	\ees
	The orthogonal complement $(V'^\perp,q|_{V'^\perp})$ is a rational quadratic subspace of $V$ of positive signature.
	Since we suppose $u\notin V'$, then $\pi_{V'^\perp}(u)\neq 0$ and~$q(\pi_{V'^\perp}(u))>0$, hence~$q(\pi_{V'}(u))<0$.
	Since there exists~$h\in H$ such that~$h(\pi_{V'}(u))$ is not a scalar multiple of~$\pi_{V'}(u)$, as one can show using reflections by suitable vectors which are not orthogonal to~$\pi_{V'}(u)$, we deduce that
	\bes
	h(u)=h(\pi_{V'}(u))+h(\pi_{V'^\perp}(u))=
	h(\pi_{V'}(u))+\pi_{V'^\perp}(u),
	\ees
	hence $h(u)$ is not a scalar multiple of $u$.
	That is, $h(u)\notin F_1$.
	
	The case $\dim(F_1)=2$ is analogous.
	Every $u\in F_1$ is such that $\pi_{V'}(u)$ lies in a negative definite quadratic subspace $W$ of $V'$ of dimension at most $2$.
	Let $h\in H$ be such that it maps $W$ to a different negative-definite subspace of $V'$.
	Then some~$u\in F_1$ is such that~$h(u)\notin F_1$.
	\end{proof}
	\section{Equidistribution results}\label{sec:equidres}
	Clozel and Ullmo~\cite{cloull;equid1} proved that any sequence of probability measures associated to strongly special subvarieties admits a convergent subsequence, and that the limit of such a subsequence is the probability measure of a strongly special subvariety; see~\cite[\textcolor{\myblack}{Th\'eor\`eme}~$1.2$]{cloull;equid1}.
	We refine \textcolor{\myblack}{such a result} in the case of orthogonal Shimura subvarieties, which are strongly special as proved in Proposition~\ref{prop;everyosv>2isstsp}, showing that also the limit measure is associated to an orthogonal Shimura subvariety;
	see Proposition~\ref{prop:imprmotocov}.

	We fix a group of \textcolor{\myblack}{isometries} $G=\SO(V,q)$ for some rational quadratic space $(V,q)$ of signature $(n,2)$, where~$n\ge 3$, a compact maximal subgroup~$K$ of $G(\RR)$, and an arithmetic subgroup~$\Gamma$ of~$G(\QQ)$.
	
	
	\subsection{Measures on orthogonal Shimura subvarieties}\label{subsec:measureonconcy}
	
	Any~$G(\RR)$-invariant Kähler metric on the symmetric domain~$\widetilde{X}=G(\RR)/K$ is a constant multiple of the metric arising from the Killing form of the Lie algebra of~$G(\RR)$.
	We choose one of those metrics, denote by~$\vol$ the associated volume form, and by~$\omega$ its induced Kähler form.
	By Wirtinger's Theorem, the volume form~$\omega^n$ is such that $\vol=\omega^n/n!$.
	Let $\mathcal{F}_{\widetilde{X}}$ be a fundamental domain of $\widetilde{X}$ with respect to the action of $\Gamma$.
	
	The restriction $\vol|_{\mathcal{F}_{\widetilde{X}}}$ induces a $G(\RR)$-invariant Kähler metric on $X$ such that $\Vol(\mathcal{F}_{\widetilde{X}})$ is finite.
	We denote by~$\nu_{\widetilde{X}}$ the normalized measure on $\widetilde{X}$ induced by the volume form
	\bes
	\frac{\vol}{\Vol(\mathcal{F}_{\widetilde{X}})}=\frac{\omega^n}{n!\Vol(\mathcal{F}_{\widetilde{X}})},
	\ees
	and by $\nu_X$ the probability measure induced on $X$ by restricting~$\nu_{\widetilde{X}}$ to~$\mathcal{F}_{\widetilde{X}}$.

	Let $Z$ be an orthogonal Shimura subvariety of $X$ of dimension $r\ge3$.
	Let~$H=\SO(V',q')$ be the subgroup of~$G$ associated to some subspace~$(V',q')$ of~$(V,q)$, such that~$Z$ is the immersion in~$X$ of the orthogonal Shimura variety~${Z'=\Gamma_H\backslash H(\RR)/K_H}$, where $\Gamma_H=\Gamma\cap H(\RR)^+$ and $K_H=K\cap H(\RR)$.
	We may rewrite such an~$r$-dimensional immersion as~${Z=\Gamma\backslash\Gamma H(\RR) K/K}$.
	We denote by~$\nu_Z$ the push-forward of the measure~$\nu_{Z'}$ via the immersion map~$Z'\to X$.
	
	\subsection{Equidistribution of orthogonal Shimura subvarieties}
	
	The following result is a refinement of~\cite[\textcolor{\myblack}{Th\'eor\`eme}~$1.2$]{cloull;equid1}.
	
	\begin{prop}\label{prop:imprmotocov}
	Let $X=\Gamma\backslash G(\RR)/K$ be an orthogonal Shimura variety, and let~$\seq{Z_m}{m}$ be a sequence of orthogonal Shimura subvarieties \textcolor{\myblack}{of~$X$} of fixed dimension.
	The sequence of probability measures~\textcolor{\myblack}{$\seq{\nu_{Z_m}}{m}$} contains a subsequence~\textcolor{\myblack}{$\seq{\nu_{Z_j}}{j}$} which weakly converges to the probability measure~\textcolor{\myblack}{$\nu_{Z}$} associated to some orthogonal Shimura subvariety~$Z$ of~$X$.
	The subvarieties~$Z_j$ are eventually contained in~$Z$.
	\end{prop}
	For the sake of brevity, whenever a sequence~$\seq{Z_j}{j}$ is such that the associated probability measures weakly converge to the one of a subvariety~$Z$, as in Proposition~\ref{prop:imprmotocov}, we say that the subvarieties~$Z_j$ \emph{equidistribute} in $Z$.
	\begin{proof}
%
%
	
	\textcolor{\myblack}{The subvarieties~$Z_m$ are strongly special by Proposition~\ref{prop;everyosv>2isstsp}.
	By \cite[Théorème $1.2$]{cloull;equid1} we deduce that there is a subsequence~$\seq{\nu_{Z_j}}{j}$ converging to a probability measure~$\nu_Z$ associated to some strongly special Shimura subvariety~$Z$.
	Furthermore, there exists~$j_0\in\NN$ such that~$Z_j\subset Z$ for every~$j\ge j_0$.
	We denote by~$M$ the~$\QQ$-algebraic subgroup of~$G$ giving rise to~$Z$, so that~$Z=\Gamma\backslash \Gamma M(\RR)K/K$.}
	
	By~\cite{fiori}, \textcolor{\myblack}{the special subvariety~$Z$} is of orthogonal type; see also Remark~\ref{rem;fioridifShimsubvar}.
	\textcolor{\myblack}{We want to prove that~${M=\SO(W,q_W)}$} for some rational quadratic subspace~$(W,q_W)$ of signature~$(r',2)$ in~$(V,q)$, where $r'\ge r$, or equivalently that~$Z$ is an orthogonal Shimura subvariety of~$X$ of dimension $r'\ge r$.
	
	Since the subvarieties $Z_j$ equidistribute in $Z$, the latter is the minimal special subvariety of $X$ containing $Z_j$ for all $j\ge j_0$.
	That is, if $Y$ is a special subvariety of $X$ containing $Z_j$ for all $j\ge j_0$, then $Y$ contains also $Z$.
	
	Let~$E_j=\SO(V_j,q_{V_j})$ be the groups of isometries of some rational quadratic subspace of signature $(r,2)$ in $(V,q)$ \textcolor{\myblack}{such that~$Z_j=\Gamma\backslash\Gamma E_j(\RR)K/K$.
	Up to conjugation of~$E_j$ by some isometry in~$\Gamma$, we may assume that~$E_j$ is a subgroup of~$M$.}
	
	Let $W$ be the rational subspace of $V$ generated by all~\textcolor{\myblack}{$V_j$} with $j\ge j_0$, and let $q_W\coloneqq q|_W$.
	Consider the orthogonal Shimura subvariety
	\bes
	Y=\Gamma\backslash\Gamma E(\RR)K/K,
	\ees
	\textcolor{\myblack}{where~$E=\SO(W,q_W)$, and let~$r'=\dim Y$}.
	By construction, we know that $Z_j\subseteq Z\subseteq Y$, and that $E_j$ is a $\QQ$-subgroup of both~$E$ and~$M$, for every $j\ge j_0$.
	
	The inclusion of $\QQ$-groups $M\hookrightarrow E$ gives rise to an immersion of Shimura varieties.
	By Remark~\ref{rem;fromFiori}, there exists a quadratic space $(U,q_U)$ over a totally real field extension~$F$ of~$\QQ$ such that~$M=\Res_{F/\QQ}\SO(U,q_U)$.
	Up to base change to $\RR$, the inclusion $M\hookrightarrow E$ factors \textcolor{\myblack}{through}
	\be\label{eq;prooffacttrFi}
	M(\RR)\lhook\joinrel\longrightarrow\SO(\ell,2)\times\SO(\ell+2)\times\cdots\times\SO(\ell+2),
	\ee
	for some $\ell\le r'$, and the projection to the first factor $\SO(\ell,2)$ is surjective.
	
	If $\ell=r'$, then there must be no compact factor $\SO(\ell+2)$ in \eqref{eq;prooffacttrFi} by dimension issues, that is,~$F=\QQ$.
	Since the projection of $M(\RR)$ to $\SO(\ell,2)$ is surjective, the inclusion~${M(\RR)\hookrightarrow E(\RR)}$ is onto, hence~$M=E$.
	
	We conclude by proving that $\ell$ can not be less than $r'$.
	We know that $M(\RR)$ contains the group of isometries $E_j(\RR)\cong \SO(r,2)$, for every $j\ge j_0$.
	The composition of the inclusion~$E_j(\RR)\hookrightarrow M(\RR)$ with~\eqref{eq;prooffacttrFi} can only land in the first factor~$\SO(\ell,2)$ of the right-hand side of~\eqref{eq;prooffacttrFi}.
	In fact, suppose that it does not. Then, projecting to one of the factors~$\SO(\ell+2)$ in~\eqref{eq;prooffacttrFi}, there exists a non-trivial homomorphism of real algebraic groups~${\phi:E_j(\RR)\to\SO(\ell+2)}$.
	Since~$\ker(\phi)$ is normal in~$E_j(\RR)$ and the latter is simple, the map~$\phi$ must be injective.
	This implies that~$E_j(\RR)$ is isomorphic to a closed subgroup in~$\SO(\ell+2)$, hence it is compact, but it is well-known that~$E_j(\RR)$ is not.
	
	Since $E_j(\RR)$ is the group of isometries of the real quadratic subspace~${W_j\otimes\RR}$ of~${W\otimes\RR}$, then~$\SO(\ell,2)$ must be the group of isometries of a real quadratic space containing such ${W_j\otimes\RR}$ for all~$j$.
	This implies that~$(\ell,2)$ must be the signature of a quadratic space containing all~$W_j$.
	Since~$W\otimes\RR$ has been chosen to be the span of the subspaces $W_j\otimes\RR$, then~$(\ell,2)$ must be at least the signature of $W\otimes\RR$, and the latter is $(r',2)$.
	This implies that~${\ell=r'}$.
	\end{proof}
	
	\subsection{Sequences of orthogonal Shimura subvarieties in cohomology}\label{sec;seqincoho}
	
	\textcolor{\myblack}{The equidistribution properties of orthogonal Shimura subvarieties illustrated in Proposition~\ref{prop:imprmotocov} enable us to deduce the asymptotic behavior of sequences of cohomology classes of such subvarieties.
	This is possible thanks to the following direct consequence of~\cite[Corollary~$1.5$]{kozmau} and~\cite[Corollary~$2.9$]{tatho}.
	\begin{thm}[Koziarz--Maubon, Tayou--Tholozan]\label{thm:mainresu2}
	Let $X$ be an orthogonal Shimura variety of dimension~$n\ge 3$, and let $\seq{Z_j}{j\in\NN}$ be a sequence of pairwise different orthogonal Shimura subvarieties in~$X$ of dimension $r\ge3$.
	If such subvarieties equidistribute in an orthogonal Shimura subvariety~$Z$ of dimension $r'>r$, then
	\be\label{eq;mainresu2}
	\frac{[Z_j]}{\Vol(Z_j)}\xrightarrow[\,j\to\infty\,]{} \frac{r!}{r'!}\cdot[\omega]^{r'-r}\wedge\frac{[Z]}{\Vol(Z)}\qquad\text{in $H^{2(n-r)}(X,\QQ)\cap H^{n-r,n-r}(X)$}.
	\ee
	\end{thm}}
	\section{Cones generated by orthogonal Shimura subvarieties}\label{sec;seqspecialcycles}
	In this section we illustrate how to use Theorem~\ref{thm:mainresu2} to compute the accumulation rays of the cone~$\coneosub{r}{X}$.
	Moreover, we illustrate some analogies of~$\coneosub{n-2}{X}$ with the cone~$\conespec{X}$ generated by special cycles of codimension~$2$; see~\cite{zufcones} for more information on~$\conespec{X}$.\\
	
	Let $X$ be a normal irreducible complex space of dimension $n$.
	A \emph{cycle $Z$ of codimension~$r$ in $X$} is a locally finite formal sum
	\bes
	Z=\sum n_Y Y,\qquad\text{$n_Y\in\ZZ$,}
	\ees
	of distinct closed irreducible analytic subsets~$Y$ of codimension $r$ in $X$.
	The \emph{support} of the cycle $Z$ is the closed analytic subset $\supp(Z)=\bigcup_{n_Y\neq 0} Y$ of pure codimension $r$.
	The integer $n_Y$ is the \emph{multiplicity} of the irreducible component $Y$ of $\supp(Z)$ in the cycle $Z$.
	
	If $X$ is a manifold, and $\Gamma$ is a group of biholomorphic transformations of $X$ acting properly discontinuously, we may consider the preimage $\pi^*(Z)$ of a cycle~$Z$ of codimension~$r$ on~$X/\Gamma$ under the canonical projection $\pi\colon X\to X/\Gamma$.
	For any irreducible component~$Y$ of~$\pi^{-1}\big(\supp(Z)\big)$, the multiplicity~$n_Y$ of~$Y$ with respect to~$\pi^*(Z)$ equals the multiplicity of~$\pi(Y)$ with respect to~$Z$.
	This implies that~$\pi^*(Z)$ is a \emph{$\Gamma$-invariant cycle}, meaning that if~$\pi^*(Z)=\sum n_Y Y$, then
	\bes
	\gamma\big(\pi^*(Z)\big)\coloneqq \sum n_Y \gamma(Y)\qquad\text{equals $\pi^*(Z)$, for every $\gamma\in \Gamma$.}
	\ees
	Note that we do not take account of possible ramifications of the cover $\pi$.\\
	
	We now focus on orthogonal Shimura varieties associated to \emph{unimodular lattices}.
	Let $L$ be an even unimodular lattice of signature $(n,2)$.
	We denote by $(\cdot{,}\cdot)$ the bilinear form of $L$, and by $q$ the quadratic form defined as $q(\cdot)=\frac{1}{2}(\cdot{,}\cdot)$.
	The $n$-dimensional complex manifold
	\bes
\mathcal{D}_n=\{ z\in L\otimes\CC\setminus\{0\}\, :\,\text{$(z,z)=0$ and $(z,\bar{z})<0$}\}/\CC^* \subset \mathbb{P}(L\otimes\CC)
	\ees
	has two connected components.
	The action of the group of the isometries of~$L$ with determinant~$1$, denoted by~$\SO(L)$, extends to an action on~$\mathcal{D}_n$.
	We choose a connected component of~$\mathcal{D}_n$ and denote it by~$\projmodn$.
	The manifold~$\projmodn$ is a model of~$G(\RR)/K$, where~${G=\SO(L\otimes\QQ)}$ and~$K$ is a compact maximal subgroup of~$G(\RR)$.
	We define~$\SO^+(L)$ as the subgroup of~$\SO(L)$ that contains all isometries which preserve~$\projmodn$.
	
	Let~$X=\Gamma\backslash\projmodn$ be the orthogonal Shimura variety arising from some finite index subgroup~$\Gamma$ of~$\SO^+(L)$.
	We denote by~${\pi\colon \projmodn\to X}$ the canonical projection.
	An attractive feature of these kind of varieties is that they have many algebraic cycles.
	We recall below the construction of the so-called \emph{special cycles}; see~\cite{ku;algcycle}.
	They are a generalization of the Heegner divisors in higher codimension; see~\cite[Section~$5$]{br;borchp} for a description of such divisors in a setting analogous to this paper.
	
	We denote by $\halfint_d$, resp.~$\halfint_d^+$, the set of symmetric half-integral positive semi-definite, resp.\ positive definite, $d\times d$-matrices.
	If $\vect{\lambda}=(\lambda_1,\dots,\lambda_d)\in L^d$, for some $d<n$, the \emph{moment matrix} of~$\vect{\lambda}$ is defined as $q(\vect{\lambda})\coloneqq\frac{1}{2}\big( (\lambda_i,\lambda_j)\big)_{i,j}$, while its orthogonal complement in~$\projmodn$ is
	\bes
	\vect{\lambda}^\perp = \bigcap_{j=1	}^d \lambda_j^\perp.
	\ees
	If $T\in\halfint^+_d$, then
	\ba\label{eq;defSpeccydimm}
	\sum_{\substack{\vect{\lambda}\in L^d\\ q(\vect{\lambda})=T}} \vect{\lambda}^\perp
	\ea
	is a $\Gamma$-invariant cycle of codimension $d$ in $\projmodn$.
	Since the componentwise action of~$\Gamma$ on the vectors $\vect{\lambda}\in L^d$ of fixed moment matrix $T\in\halfint^+_d$ has finitely many orbits, the cycle~\eqref{eq;defSpeccydimm} descends to a cycle of codimension~$d$ on~$X$, which we denote by $Z(T)$ and call \emph{the special cycle associated to $T$}.
	The special cycles of codimension~$1$ are usually called \emph{Heegner divisors} and denoted by~$H_m$, where~$m\in\ZZ_{>0}$. 
	\begin{rem}
	The group $\GL_d(\ZZ)$ acts on $\halfint_d$ preserving~$\halfint^+_d$ under the action $T\mapsto u^t\cdot T\cdot u$, where~${u\in\GL_d(\ZZ)}$ and $T\in\halfint_d$.
	Since $q(u\cdot\vect{\lambda}^t)=u\cdot q(\vect{\lambda})\cdot u^t$ for every $u\in\GL_d(\ZZ)$ and~$\vect{\lambda}\in L^d$ with $q(\vect{\lambda})\in\halfint_d^+$, it is easy to see that $Z(T)=Z(u^t\cdot T\cdot u)$.
	\end{rem}
		
	Suppose now that~$\Gamma=\SO^+(L)$.
	If $m$ is a positive integer, we denote by $\primHeegner{m}$ the $m$-th primitive Heegner divisor, that is, the divisor of $X$ descending from the~$\Gamma$-invariant divisor of~$\projmodn$ defined as
	\be\label{eq;sumforprimHeegner}
	\sum_{\substack{\lambda\in L\,\text{primitive}\\ q(\lambda)=m}}\lambda^\perp.
	\ee
	\begin{rem}[See {\cite[Lemma~$4.2$ and~($17$)]{brmo}}]\label{rem;splitHeegninprinHeegn}
	If $m$ is squarefree, then the Heegner divisor~$H_m$ is the same as the primitive Heegner divisor~$\primHeegner{m}$.
	In general, we have
	\bes
	H_m=\sum_{\substack{t\in\ZZ_{>0}\\ t^2|m}}\primHeegner{m/t^2}
	\qquad\text{and}\qquad
	\primHeegner{m}=\sum_{\substack{t\in\ZZ_{>0}\\ t^2|m}} \mu(t) H_{m/t^2}
	,
	\ees
	where~$\mu$ is the Möbius function.
	\end{rem}
	
	In the following result we gather some basic properties of the irreducible components of special cycles.
	\begin{lemma}\label{lemma;basicsspcy}
	Let~$X=\SO^+(L)\backslash G(\RR)/K$ be an orthogonal Shimura variety of dimension~$n>2$ arising from an even unimodular lattice~$L$.
	\begin{enumerate}[label=(\roman*)]
	\item\label{item1;basicsspcy}
	All irreducible components of~$Z(T)$, where~$T\in\halfint^+_d$, are orthogonal Shimura subvarieties of codimension $d$ in $X$, and all orthogonal Shimura subvarieties of codimension~$d$ in~$X$ arise in this way,
	\item For every positive integer $m$, we have $\primHeegner{m}=2\cdot\Gamma\backslash\Gamma \lambda^\perp$, where $\lambda\in L$ is any primitive lattice vector such that $q(\lambda)=m$.
	Equivalently,~$\primHeegner{m}$ is the orthogonal Shimura subvariety $\Gamma\backslash\Gamma \lambda^\perp$ of $X$ counted twice.\label{item2;basicsspcy}
	\item Let~$T=(m_{i,j})_{i,j}\in\halfint^+_d$ be such that~$m_{\ell,\ell}$ is squarefree for some index~$\ell$.\label{item3;basicsspcy}
	All irreducible components of $Z(T)$ are subvarieties of the irreducible component of~$H_{m_{\ell,\ell}}$.
	\end{enumerate}
	\end{lemma}
	\begin{proof}
	We begin with~\ref{item1;basicsspcy}.
	It is easy to see that every irreducible component of~$Z(T)$ is by definition the immersion in~$X$ of the orthogonal Shimura variety associated to the~$\QQ$-subgroup~$H=\SO(\langle \lambda_1,\dots,\lambda_d\rangle_\QQ^\perp)$ of $G=\SO(L\otimes\QQ)$, and to the arithmetic group~${\Gamma\cap H(\QQ)^+}$.
	In fact, the quadratic subspace~$\langle \lambda_1,\dots,\lambda_d\rangle_\QQ^\perp$ of $L\otimes\QQ$ is of signature~$(n-d,2)$, since~$T$ is non-singular.
	Conversely, if~$Z$ is an orthogonal Shimura subvariety of codimension~$d$ in~$X$, then it arises from a rational quadratic subspace~$(V',q')$ of signature~$(n-d,2)$ in~$(V,q)$, where~$V=L\otimes\QQ$.
	Let~$S$ be the orthogonal complement of~$(V',q')$ in~$(V,q)$.
	It is a rational quadratic space of signature~$(d,0)$.
	Let~$v_1,\dots,v_d$ be a basis of~$S$.
	Up to multiplying by suitable integers, we may assume that such a basis is made of lattice vectors of~$L$.
	This implies that~$Z$ is an irreducible component of the special cycles~$Z(q(v_1,\dots,v_d))$.
	
	Point~\ref{item2;basicsspcy} is~\cite[Lemma~$4.3$]{brmo}, we briefly recall the proof.
	Since~$q(\lambda)=q(-\lambda)$ and~$\lambda^\perp=(-\lambda)^\perp$, we see that in~\eqref{eq;sumforprimHeegner} every subvariety $\lambda^\perp$, such that $\lambda$ is primitive with norm~${q(\lambda)=m}$, is counted twice.
	In fact, the only primitive lattice vectors of norm~$m$ generating the line $\RR\cdot\lambda\subset L\otimes\RR$ are~$\lambda$ and~$-\lambda$.
	By~\cite[Lemma~$4.4$]{FreiHerm}, any two primitive lattice vectors in~$L$ with the same norm lie in the same~$\SO^+(L)$-orbit.
	This implies that~${\Gamma\lambda^\perp=\Gamma\lambda'^\perp}$, for every primitive~${\lambda,\lambda'\in L}$ of norm~$m$.
	
	We conclude with~\ref{item3;basicsspcy}.
	Let $\vect{\lambda}=(\lambda_1,\dots,\lambda_d)\in L^d$ be such that $q(\vect{\lambda})=T$.
	If~$m_{\ell,\ell}$ is squarefree, then the entry~$\lambda_\ell$ is a primitive lattice vector of~$L$.
	By~\ref{item2;basicsspcy}, we deduce that~$\Gamma\backslash\Gamma\lambda_\ell^\perp$ is the irreducible component of the Heegner divisor $\primHeegner{m_{\ell,\ell}}$ on~$X$.
	Since $\vect{\lambda}^\perp$ is contained in~$\lambda_\ell^\perp$, then~$\Gamma\backslash\Gamma \vect{\lambda}^\perp$ is a subvariety of~$\Gamma\backslash\Gamma\lambda_\ell^\perp$.	
	\end{proof}
	
	\subsection{Proof of Theorem~\ref{thm;acconeortshsubv}}
	In this section we prove Theorem~\ref{thm;acconeortshsubv}, which is stated below for the convenience of the reader.
	\begin{thm}\label{thm;againfromintro}
	Let~$X$ be an orthogonal Shimura variety of dimension~$n$, and let~$r>(n+1)/2$.
	If~$[Z]$ is a non-zero cohomology class arising from an orthogonal Shimura subvariety~$Z$ of dimension~$r'>r$ in~$X$, then the ray~$\RR_{\ge0}\cdot [\omega]^{r'-r}\wedge[Z]$ is an accumulation ray of~$\coneosub{r}{X}$.
	\end{thm}
	\begin{rem}
	It is well-known that the cohomology classes of compact subvarieties in a Kähler manifold are non-zero.
	\textcolor{\myblack}{As already mentioned in the introduction,} this result is not available in the literature for (non-compact) orthogonal Shimura subvarieties.
	For this reason, the hypothesis appearing in Theorem~\ref{thm;againfromintro} that the class~$[Z]$ is non-zero is not trivial.
	\end{rem}
	\begin{rem}
	In~\cite{fm;boundarybe} it is proved that infinitely many special cycles (of fixed dimension) have non-vanishing cohomology classes.
	By Theorem~\ref{thm;againfromintro}, we may deduce an analogous statement for the classes of irreducible components of special cycles.
	In fact, if~${R\coloneqq \RR_{\ge0}\cdot [\omega]^{r'-r}\wedge[Z]}$ is an accumulation ray of~$\coneosub{r}{X}$, then there are \emph{infinitely many non-zero} cohomology classes of irreducible components of special cycles whose associated rays converge to~$R$.
	\end{rem}
	To prove Theorem~\ref{thm;againfromintro} we need the following auxiliary results.
	\begin{lemma}\label{lem;exofequidsubvarinZ}
	Let~$Z$ be an orthogonal Shimura subvariety of dimension~$r'\ge 2$ in~$X$, and let~$r<r'$ be a positive integer.
	There exists a sequence of pairwise different orthogonal Shimura subvarieties~$\seq{Z_j}{j}$ of dimension~$r$ that equidistribute in~$Z$.
	\end{lemma}
	\begin{proof}
	Let~$G=\SO(V,q)$ for some rational quadratic space~$(V,q)$ of signature~$(n,2)$, and let~$(W,q_W)$ be a rational quadratic subspace of signature~$(r',2)$ in~$(V,q)$ such that
	\[
	Z=\Gamma\backslash\Gamma H(\RR) K/K,\qquad \text{where~$H=\SO(W,q_W)$}.
	\]
	
	Suppose that~$r=r'-1$.
	We are going to construct pairwise different subvarieties
	\[
	Z_j=\Gamma\backslash\Gamma H_j(\RR) K/K
	\]
	arising from~$H_j=\SO(V_j,q_j)$, for some rational quadratic subspaces~$(V_j,q_j)$ of signature~${(r,2)}$ in~$(W,q_W)$.
	Since they have codimension~$1$ in~$Z$, they equidistribute in~$Z$ by Proposition~\ref{prop:imprmotocov}.
	
	Let~$(V_0,q_0)$ be any subspace of signature~${(r,2)}$ in~$(W,q_W)$, and let~$Z_0$ be the orthogonal Shimura subvariety arising from~$H_0=\SO(V_0,q_0)$.
	By Lemma~\ref{lemma:gaghk}, every~$r$-dimensional orthogonal Shimura subvariety of~$X$ is of the form
	\[
	\Gamma\backslash\Gamma g H_0(\RR) K/K,\qquad\text{for some~$g\in G(\RR)$.}
	\]
	We choose~$Z_1$ to be the subvariety arising from some~$g=g_1\in H(\QQ)$ such that~$g_1\not\in\Gamma H_0(\QQ)$.
	Note that~$Z_1\subseteq Z$ and~$Z_1\neq Z_0$.
	
	We iterate this construction such that for every~$j\in\ZZ_{>0}$, the subvariety~$Z_j$ arises from some~$g=g_j\in H(\QQ)$ such that~$g_j$ does not lie in
	\[
	\Gamma H_0(\QQ),\qquad \Gamma g_1 H_0(\QQ),\qquad \dots,\qquad \Gamma g_{j-1} H_0(\QQ).
	\]
	In this way we obtain a sequence of pairwise different subvarieties~$\seq{Z_j}{j}$ as requested, concluding the proof if~$r=r'-1$.
	
	We proceed by reverse induction on~$r$.
	Suppose that there exists a sequence of pairwise different orthogonal Shimura subvarieties~$\seq{Z_j}{j}$ of dimension~$r+1$ that equidistribute in~$Z$.
	Following the same construction as above, there exists a sequence of pairwise different orthogonal Shimura subvarieties~$\seq{Z_{j,i}}{i}$ of dimension~$r$ which equidistribute in~$Z_j$, for every~$j$.
	Consider the sequence~$\seq{Z_{j,j}}{j}$.
	Without loss of generality we may assume that the subvarieties~$Z_{j,j}$ are pairwise different.
	Furthermore, we may assume that they are not eventually contained in any proper subvariety of~$Z$, since the~$Z_j$ equidistribute in~$Z$.
	
	The subvarieties~$Z_{j,j}$ equidistribute in~$Z$ by Proposition~\ref{prop:imprmotocov}.
	\end{proof}	
	The following ancillary result can be easily proved in the same way as~\cite[Proof of Corollary~$2.3$]{MoTo}.
	\begin{lemma}\label{lem;voldiverge}
	Let~$(Z_j)_j$ be a sequence of pairwise different orthogonal Shimura subvarieties of the same dimension.
	The volumes~$\Vol(Z_j)$ diverge as~$j\to\infty$.
	\end{lemma}
	We are now ready to prove the main result of this section.
	\begin{proof}[Proof of Theorem~\ref{thm;againfromintro}]
	By Lemma~\ref{lem;exofequidsubvarinZ} there exists a sequence of pairwise different orthogonal Shimura subvarieties~$\seq{Z_j}{j}$ of dimension~$r$ that equidistribute in~$Z$.
	By \textcolor{\myblack}{Theorem~\ref{thm:mainresu2}} we deduce that
	\[
	\frac{[Z_j]}{\Vol(Z_j)}\xrightarrow[\,j\to\infty\,]{} \frac{r!}{r'!}\cdot[\omega]^{r'-r}\wedge\frac{[Z]}{\Vol(Z)}\qquad\text{in $H^{2(n-r)}(X,\RR)$}.
	\]
	
	The cohomology~$H^{s}(X,\CC)$ is isomorphic to the intersection cohomology~$IH^s(\bbcomp,\CC)$ of the Baily--Borel compactification~$\bbcomp$ of~$X$, for every~$s<n-1$.
	Clearly, this isomorphism of cohomologies is available if~$s=2(n-r)$ under the assumption that~$r>(n+1)/2$. 
	We may then deduce that the map given by wedging classes in~$H^{2r'}(X,\CC)$ by~$[\omega]^{r'-r}$ is injective by the Hard-Lefschetz Theorem; see~\cite[ Remark~$4.8$]{zufcones} and~\cite[Corollary~$9.2.3$]{interscoho}.
	Since~$[Z]\neq 0$, we deduce that~$[\omega]^{r'-r}\wedge [Z]\neq 0$. 
	
	Since~$\Vol(Z_j)$ diverges by Lemma~\ref{lem;voldiverge}, we may assume that there exists a subsequence of subvarieties~$\seq{Z_i}{i}$ whose cohomology classes are pairwise different and non-zero.
	Since
	\[
	\RR_{\ge0}\cdot [Z_j] \xrightarrow[\,j\to\infty\,]{} \RR_{\ge0}\cdot[\omega]^{r'-r}\wedge[Z],
	\]
	we deduce that $\RR_{\ge0}\cdot [\omega]^{r'-r}\wedge[Z]$ is an accumulation ray of~$\coneosub{r}{X}$.
	\end{proof}
	
	\subsection{Proof of Corollary~\ref{cor;acconeortshsubv}}
	We prove here  Corollary~\ref{cor;acconeortshsubv}, which we restate below for convenience.
	\begin{cor}
	Let~$X=\SO^+(L)\backslash G(\RR)/K$ be an orthogonal Shimura variety of dimension~$n>5$ arising from an even unimodular lattice~$L$.
	The accumulation rays of~$\coneosub{n-2}{X}$ are~$\RR_{\ge0}\cdot[\omega]^2$ and the rays~$\RR_{\ge0}\cdot[\omega]\wedge[Z]$, for every orthogonal Shimura subvariety~$Z$ of codimension~$1$.
	\end{cor}
	\begin{proof}
	Let~$\seq{Z_j}{j}$ be a sequence of pairwise different orthogonal Shimura subvarieties of codimension~$2$ in~$X$.
	By \textcolor{\myblack}{Theorem~\ref{thm:mainresu2}}, up to extracting a subsequence, the sequence $[Z_j]/\Vol(Z_j)$ converges towards a positive multiple of either~$[\omega]^2$ or~$[\omega]\wedge[Z]$, for some orthogonal Shimura subvariety~$Z$ of codimension~$1$.
	This implies that all accumulation rays are generated by such kind of cohomology classes.
	To show that all such cohomology classes generate an accumulation ray, we need to show that they do not vanish, so that we may apply Theorem~\ref{thm;againfromintro}.
	By e.g.~\cite[Remark~$4.8$]{zufcones}, since the map~$H^2(X,\CC)\to H^4(X,\CC)$ given by wedging with the Kähler class~$[\omega]$ is injective, it is enough to show that~$[\omega]$ and~$[Z]$ are non-zero in~$H^2(X,\RR)$.
	
	Let~$M^{1+n/2}(\RR)$ be the space of elliptic modular forms of weight~$1+n/2$ with real Fourier coefficients.
	The dual space~$M^{1+n/2}(\RR)^*$ is generated by the \emph{coefficient extraction functionals}~$c_m$, which extract from every modular form~$f\in M^{1+n/2}(\RR)$ its $m$-th Fourier coefficient~$c_m(f)$.
	
	Let~$H_m$ be the~$m$-th Heegner divisor of~$X$.
	In~\cite{brmo} Bruinier and Möller used the injectivity of the Kudla--Millson lift, see e.g.~\cite{br;converse} \cite{zufunfolding}, to show that the map
	\[
	\psi\colon M^{1+n/2}(\RR)^*\to H^2(X,\RR),\qquad c_0\mapsto -[\omega]\quad\text{and}\quad c_m\mapsto [H_m]\quad\text{for all~$m\in\ZZ_{>0}$,}
	\]
	is injective.
	By Lemma~\ref{lemma;basicsspcy}, every orthogonal Shimura subvariety~$Z$ of codimension~$1$ is (half of) a primitive Heegner divisor~$\primHeegner{m}$, for some positive integer~$m$.
	It is then enough to show that~$\psi^{-1}([\primHeegner{m}])\neq 0$, for every~$m$.
	
	By Remark~\ref{rem;splitHeegninprinHeegn}, we may compute
	\be\label{eq;psiminusprimheegner}
	\psi^{-1}([\primHeegner{m}])=
	\sum_{t^2|m}\mu(t)\cdot\psi^{-1}([H_{m/t^2}])=
	\sum_{t^2|m}\mu(t)\cdot c_{m/t^2}.
	\ee
	Let~$E_{1+n/2}\in M^{1+n/2}(\RR)$ be the (normalized) Eisenstein series of weight~$1+n/2$.
	It is well-known that
	\[
	c_m(E_{1+n/2})=\frac{2\sigma_{n/2}(m)}{\zeta(-n/2)},
	\]
	for all~$m\in\ZZ_{>0}$, where~$\zeta$ is the Riemann zeta function and~$\sigma_{n/2}(m)$ is the sum of the~$n/2$-powers of the positive divisors of~$m$.
	
	By~\eqref{eq;psiminusprimheegner}, if we evaluate~$\psi^{-1}([\primHeegner{m}])$ on the Eisenstein series~$E_{1+n/2}$ we obtain
	\ba\label{eq;comofpsi-oneis}
	\psi^{-1}([\primHeegner{m}])\big(E_{1+n/2}\big)=
	\sum_{t^2|m}\mu(t)\cdot c_{m/t^2}(E_{1+n/2})=\qquad\qquad\qquad\\
	=\frac{2}{\zeta(-n/2)}\sum_{t^2|m}\mu(t)\cdot\sigma_{n/2}(m/t^2)=
	\frac{2m^{n/2}}{\zeta(-n/2)}\prod_{p|m}(1+p^{-n/2}),
	\ea
	where the last equality follows from e.g.~\cite[p.~$352$]{boda;petnorm} and~\cite[Section~$3$]{hayashida}.
	Since the right-hand side of~\eqref{eq;comofpsi-oneis} is non-zero, we deduce that the functional~$\psi^{-1}([\primHeegner{m}])$ is non-zero.
	\end{proof}
	
	\subsection{Comparison with cones of special cycles}
	
	In this section we illustrate how to use the equidistribution results of this paper to deduce properties on cones of special cycles.
	Descriptions of the cones of special cycles of codimension~$1$ and~$2$ can be found respectively in~\cite{brmo} and~\cite{zufcones}.\\
	
	The next proposition is~\cite[Proposition~$4.5$]{brmo}, therein proved in terms of modular forms using the modularity of Kudla's generating series of Heegner divisors.
	We provide here a different proof in terms of equidistribution.
	\begin{prop}[Bruinier--Möller]\label{prop;JMconvprimheegn}
	Let~$X=\SO^+(L)\backslash G(\RR)/K$ be an orthogonal Shimura variety of dimension $n>5$ arising from an even unimodular lattice~$L$.
	We have
	\be\label{eq;JMconvprimheegn}
	\RR_{\ge0}\cdot[\primHeegner{m}] \xrightarrow[m\to\infty]{}\RR_{\ge0}\cdot [\omega]\qquad\text{in $H^2(X,\RR)$.}
	\ee
	\end{prop}
	\begin{proof}
	As illustrated in Lemma~\ref{lemma;basicsspcy} \ref{item2;basicsspcy}, the primitive Heegner divisor~$\primHeegner{m}$ is twice an orthogonal Shimura variety of the form $\Gamma\backslash\Gamma\lambda^\perp$, for some primitive lattice vector $\lambda\in L$ such that $q(\lambda)=m$.
	Since any lattice vector can be written uniquely as a positive multiple of a primitive lattice vector, so that the only primitive lattice vectors in $L$ generating the line~${\RR\cdot \lambda\subset L\otimes\RR}$ are $\lambda$ and $-\lambda$, we deduce that the irreducible components of the divisors in the sequence~$\seq{\primHeegner{m}}{m\in\NN}$ are pairwise different.
	By Proposition~\ref{prop:imprmotocov}, there is no subsequence of~$\seq{\primHeegner{m}}{m\in\NN}$ without convergent subsequences.
	Since the $\primHeegner{m}$ are pairwise different of codimension~$1$ in~$X$, we deduce that the only subvariety of $X$ in which the~$\primHeegner{m}$ can equidistribute is~$X$ itself.
	We then deduce~\eqref{eq;JMconvprimheegn} from \textcolor{\myblack}{Theorem~\ref{thm:mainresu2}}.
	\end{proof}
	
	In~\cite{zufcones} we computed the accumulation rays of the cone~$\conespec{X}$ of codimension~$2$ special cycles of~$X$.
	We also proved that the associated accumulation cone is rational and polyhedral.
	It is natural to ask whether the same properties are satisfied also by the larger cone generated by the orthogonal Shimura subvarieties of codimension~$2$.
	The answer to this question is given by Corollary~\ref{cor;acconeortshsubvfoll}, which we restate here for convenience.
	\begin{cor}
	Let~$X=\SO^+(L)\backslash G(\RR)/K$ be an orthogonal Shimura variety of dimension~$n>5$ arising from an even unimodular lattice~$L$.
	The accumulation cone of~$\coneosub{n-2}{X}$ equals the accumulation cone of~$\conespec{X}$.
	In particular, the accumulation cone of~$\coneosub{n-2}{X}$ is pointed, rational, polyhedral and of dimension~$\dim M^{1+n/2}_1$.
	\end{cor}
	\begin{proof}
	The classification of the accumulation rays of~$\coneosub{n-2}{X}$ is provided in Corollary~\ref{cor;acconeortshsubv}, while the classification of the accumulation rays of~$\conespec{X}$ is provided in~\cite[Corollary~$8.3$]{zufcones}.
	Note that the latter result is valid also in cohomology, and not only in the Chow group~$\CH^2(X)\otimes\RR$.
	It is then clear that the two accumulation cones are equal.
	The geometric properties of the accumulation cone follow from~\cite[Section~$6$]{zufcones}.
	\end{proof}
	
	Recall that the orthogonal Shimura subvarieties are irreducible components of special cycles; see Lemma~\ref{lemma;basicsspcy}.
	In the following result we illustrate how to compute the accumulation rays arising from sequences of orthogonal Shimura subvarieties extracted from a sequence of codimension~$2$ special cycles.
	\begin{prop}\label{prop;ressubseqHeegner}
	Let~$X=\SO^+(L)\backslash G(\RR)/K$ be an orthogonal Shimura variety of dimension~$n>5$ arising from an even unimodular lattice~$L$.
	Let~$\seq{T_j}{j\in\NN}$ be a sequence of matrices~$T_j=\big(\begin{smallmatrix}
	n_j & r_j/2\\ r_j/2 & m
	\end{smallmatrix}\big)$ in~$\halfint^+_2$ of increasing determinant.
	Let $\seq{Z_j}{j\in\NN}$ be a sequence of \emph{pairwise different} subvarieties of~$X$, chosen such that~$Z_j$ is one of the irreducible components of the special cycle~$Z(T_j)$, for every~$j$.
	\begin{enumerate}[label=(\roman*)]
	\item\label{item1;ressubseqHeegner}
	If~$m$ is squarefree, then
	\bes
	\RR_{\ge0}\cdot[Z_j] \xrightarrow[j\to\infty]{}\RR_{\ge0}\cdot [H_m]\wedge[\omega]\qquad\text{in $H^4(X,\RR)$.}
	\ees
	\item\label{item2;ressubseqHeegner}
	If $m$ is non-squarefree, then there exists a square-divisor~$t$ of~$m$, and a subsequence~$\seq{Z_s}{s}$, such that
	\bes
	\RR_{\ge0}\cdot[Z_s] \xrightarrow[s\to\infty]{}\RR_{\ge0}\cdot [\primHeegner{m/t^2}]\wedge[\omega]\qquad\text{in $H^4(X,\RR)$.}
	\ees
	\end{enumerate}
	\end{prop}
	\begin{proof}
	We begin with~\ref{item1;ressubseqHeegner}.
	By Proposition~\ref{prop:imprmotocov}, there exists a subsequence $\seq{Z_s}{s}$ of~$\seq{Z_j}{j\in\NN}$, and an orthogonal Shimura subvariety~$Z$ of dimension~${r'>n-2}$ in~$X$, such that the~$Z_s$ equidistribute in~$Z$, in particular~${Z_s\subseteq Z}$ for every~$s$ large enough.
	By Lemma~\ref{lemma;basicsspcy}~\ref{item3;basicsspcy}, all~$Z_s$ are codimension~$1$ subvarieties of the irreducible component of the Heegner divisor~$H_m$.
	This implies that~$Z$ is such irreducible component, and~$r'=n-1$.
	By \textcolor{\myblack}{Theorem~\ref{thm:mainresu2}} we deduce that
	\be\label{eq;prooffirstcod2}
	\frac{[Z_s]}{\Vol(Z_s)} \xrightarrow[s\to\infty]{}\frac{(n-2)!}{r!} [\omega]^{r'-(n-2)}\wedge\frac{[Z]}{\Vol(Z)}\qquad\text{in $H^4(X,\RR)$.}
	\ee
	We know from Lemma~\ref{lemma;basicsspcy}~\ref{item2;basicsspcy} that~$H_m=2Z$.
	Since the volume of a subvariety is non-negative, we deduce that the sequence of rays in~$\coneosub{n-2}{X}$ generated by the cohomology classes appearing in~\eqref{eq;prooffirstcod2} is such that
	\be\label{eq;subseqconvraycodim2}
	\RR_{\ge0}\cdot[Z_s] \xrightarrow[s\to\infty]{}\RR_{\ge0}\cdot [H_m]\wedge[\omega]\qquad\text{in $H^4(X,\RR)$.}
	\ee
	By Proposition~\ref{prop:imprmotocov} there is no subsequence of~$\seq{Z_j}{j\in\NN}$ without equidistributing subsequences.
	Since the~$Z_j$ are pairwise different, and since~$Z$ is the only subvariety of~$X$ in which any subsequence of $\seq{Z_j}{j\in\NN}$ can equidistribute, we deduce that~\eqref{eq;subseqconvraycodim2} is satisfied by the whole~$\seq{Z_j}{j\in\NN}$.
	
	We now prove~\ref{item2;ressubseqHeegner}.
	By Proposition~\ref{prop:imprmotocov} there exists a subsequence~$\seq{Z_s}{s}$ as above and an orthogonal Shimura subvariety~$Z$ in which the~$Z_s$ equidistribute.	
	By construction, all irreducible components of the special cycles~$Z(T_j)$ are contained in~${\Gamma\backslash\Gamma\lambda_j^\perp}$, for some~${\lambda_j\in L}$ such that~${q(\lambda_j)=m}$.
	Let $t'_j\in\ZZ_{>0}$ be such that $\lambda_j'\coloneqq \lambda_j/t'_j$ is a primitive lattice vector in $L$, so that ${t'_j}^2$ divides $m$.
	By Lemma~\ref{lemma;basicsspcy} \ref{item2;basicsspcy}, we deduce that $\Gamma\backslash\Gamma \lambda_j^\perp=\Gamma\backslash\Gamma\lambda_j'^\perp$ is the irreducible component of $\primHeegner{m/{t'_j}^2}$.
	Since the number of such primitive Heegner divisors is finite, there exists a square divisor $t$ of $m$ such that, up to extracting a subsequence, all~$Z_s$ are subvarieties of~$\primHeegner{m/t^2}$.
	Since the~$Z_s$ have codimension~$1$ in~$\primHeegner{m/t^2}$, then the latter is the only subvariety in which the~$Z_s$ can equidistribute.
	This means that $Z=\primHeegner{m/t^2}$.
	\textcolor{\myblack}{Theorem~\ref{thm:mainresu2}} concludes the proof.	
	\end{proof}
	
	\begin{rem}\label{rem;manycompinccycles}
	In Proposition~\ref{prop;ressubseqHeegner} the hypothesis that the subvarieties~$Z_j$ are \emph{pairwise different} can not be dropped.
	In fact, as illustrated in Example~\ref{ex;commonirredcomp}, it is possible to construct a sequence of matrices $T_j=\big(\begin{smallmatrix}
	n_j & r_j/2\\
	r_j/2 & m
	\end{smallmatrix}\big)$ of increasing determinant such that all special cycles $Z(T_j)$ have a common irreducible component, for every positive $m$.
	\end{rem}
	\begin{ex}\label{ex;commonirredcomp}
	Let $m$ be a positive integer, and let $L$ be a unimodular lattice of signature~$(n,2)$ such that $n>2$.
	Choose $\lambda_1,\lambda_2\in L$ to be orthogonal lattice vectors such that~${q(\lambda_1)>0}$ and ${q(\lambda_2)=m}$, and consider the matrices $T_j=\big(\begin{smallmatrix}
	j^2\cdot q(\lambda_1) & 0\\
	0 & m
	\end{smallmatrix}\big)\in\halfint^+_2$, for every~${j\in\NN}$.
	All special cycles $Z(T_j)$ have the subvariety~$Y\coloneqq\Gamma\backslash\Gamma\vect{\lambda}^\perp$ as common irreducible component, where~${\vect{\lambda}\coloneqq(\lambda_1,\lambda_2)}$.
	In fact, if we choose $\vect{\lambda}_j\coloneqq (j\lambda_1,\lambda_2)\in L^2$ for every $j$, then~${q(\vect{\lambda}_j)=T_j}$, and since $\vect{\lambda}^\perp=\vect{\lambda}_j^\perp$ as submanifolds in~$\projmodn$, we deduce that $Y$ is common to every $Z(T_j)$.
	\end{ex}
	
	In~\cite{brmo}, the convergence of~\eqref{eq;JMconvprimheegn} in Proposition~\ref{prop;JMconvprimheegn} is proven also if the primitive Heegner divisors~$\primHeegner{m}$ are replaced by the Heegner divisors~$H_m$.
	Proposition~\ref{prop:imprmotocov} and \textcolor{\myblack}{Theorem~\ref{thm:mainresu2}} do not immediately imply such result.
	In fact, since $H_m$ has, for non-squarefree~$m$, many different irreducible components which are primitive Heegner divisors associated to smaller indexes, in the sequence~$\seq{H_m}{m\in\NN}$ the divisors have many irreducible components which repeatedly appear.
	To deduce the generalization of~\cite{brmo} explained above, one should prove that such repeated components does not play any role in the convergence of the sequence~${\seq{\RR_{>0}\cdot[H_m]}{m\in\NN}}$, more precisely that
	\be\label{eq;minpartnoroleconv}
	\sum_{\substack{t^2|m\\ t>1}}\frac{[\primHeegner{m/t^2}]}{\Vol(\primHeegner{m})}\xrightarrow[m\to\infty]{} 0\qquad\text{in $H^2(X,\RR)$.}
	\ee
	
	In~\cite[Section~$8$]{zufcones} we explained that sequences of rays generated by special cycles of codimension~$2$ associated to reduced matrices of increasing determinant may have many different accumulation rays, and we computed all of them.
	For instance, if we choose~$T_j$ as in Proposition~\ref{prop;ressubseqHeegner}~\ref{item1;ressubseqHeegner}, i.e.~$T_j=\big(\begin{smallmatrix}
	n_j & r_j/2\\
	r_j/2 & m
	\end{smallmatrix}\big)\in\halfint^+_2$ is reduced with~$m$ squarefree, then
	\be\label{eq;accrayscyceas}
	\RR_{\ge0}\cdot [Z(T_j)] \xrightarrow[j\to\infty]{} \RR_{\ge0}\cdot [H_m]\wedge [\omega].
	\ee
	This was proved by means of Fourier coefficients of Siegel modular forms.
	As for the case of Heegner divisors, Proposition~\ref{prop:imprmotocov} and \textcolor{\myblack}{Theorem~\ref{thm:mainresu2}} do not immediately imply~\eqref{eq;accrayscyceas}, since in the cycles~$Z(T_j)$ have in general many irreducible components which repeatedly appear; see Remark~\ref{rem;manycompinccycles}.
	As above, to deduce~\eqref{eq;accrayscyceas} one should prove that such repeated components does not play any role in the convergence of the sequence~${\seq{\RR_{>0}\cdot [Z(T_j)]}{j\in\NN}}$.
	\\
	
	\subsection*{Data availability Statement} Data sharing not applicable to this article as no datasets were generated or analyzed during the current study.
	\\
	
	The corresponding author states that there is no conflict of interest.
	
	\printbibliography

\end{document}